\documentclass{article}

\usepackage{amsmath,amsfonts,amssymb,amsthm}
\usepackage{graphicx}
\usepackage{xcolor}
\usepackage{verbatim}
\usepackage[title]{appendix}
\usepackage{paralist}
\usepackage{url}
\usepackage{mathtools}
\newtheorem{theorem}{Theorem}
\newtheorem{lemma}{Lemma}
\newtheorem{corollary}{Corollary}
\theoremstyle{definition}
\newtheorem{remark}{Remark}

\renewcommand{\le}{\leqslant}
\renewcommand{\ge}{\geqslant}
\renewcommand{\leq}{\leqslant}
\renewcommand{\geq}{\geqslant}
\renewcommand{\emptyset}{\varnothing}

\newcommand{\ca}{\mathcal{A}}
\newcommand{\ce}{\mathcal{E}}
\newcommand{\tmod}{\ \mathsf{mod}\ }

\newcommand{\real}{\mathbb{R}}
\newcommand{\ints}{\mathbb{Z}}
\newcommand{\natu}{\mathbb{N}}

\newcommand{\caln}{\mathcal{N}}

\newcommand{\bsa}{\boldsymbol{a}}
\newcommand{\bsc}{\boldsymbol{c}}
\newcommand{\bsk}{\boldsymbol{k}}

\newcommand{\bsv}{\boldsymbol{v}}
\newcommand{\bsx}{\boldsymbol{x}}
\newcommand{\bst}{\boldsymbol{t}}

\newcommand{\bszero}{\boldsymbol{0}}
\newcommand{\bsone}{\boldsymbol{1}}

\newcommand{\bsM}{\boldsymbol{M}}
\newcommand{\bskappa}{\boldsymbol{\kappa}}

\newcommand{\rd}{\,\mathrm{d}}

\newcommand{\dunif}{\mathbb{U}}

\newcommand{\e}{\mathbb{E}}
\newcommand{\var}{\mathrm{Var}}

\newcommand{\rank}{\mathrm{rank}}
\newcommand{\row}{\mathrm{Row}}

\newcommand{\giv}{\!\mid\!}

\newcommand{\tran}{\mathsf{T}}

\newcommand{\walk}{\mathrm{wal}_k}
\newcommand{\walbsk}{\mathrm{wal}_{\bsk}}

\newcommand{\ch}{\mathcal{H}}

\newcommand{\hk}{\mathrm{HK}}
\newcommand{\re}{\mathrm{Re}}

\newcommand{\supp}{\boldsymbol{s}}

\newcommand{\newconstant}{D}

\begin{document}

\title{Super-polynomial accuracy of multidimensional randomized nets
using the median-of-means}
\author{Zexin Pan \\ Stanford University
        \and    Art B. Owen \\ Stanford University}

\date{August 2022}
\maketitle
\begin{abstract}
We study approximate integration of a
function $f$ over $[0,1]^s$
based on taking the median of $2r-1$ integral
estimates derived from independently
randomized $(t,m,s)$-nets in base $2$.  The nets
are randomized by Matousek's random linear scramble
with a digital shift.  If $f$ is analytic over
$[0,1]^s$, then the probability that any one randomized
net's estimate has an error larger
than $2^{-cm^2/s}$ times a quantity depending
on $f$ is $O(1/\sqrt{m})$ for any
$c<3\log(2)/\pi^2\approx 0.21$.
As a result the median of the distribution of
these scrambled nets has an error that is
$O(n^{-c\log(n)/s})$ for $n=2^m$ function evaluations.
The sample median of $2r-1$ independent draws attains
this rate too, so long as $r/m^2$ is bounded away from
zero as $m\to\infty$. We include results for finite
precision estimates and some non-asymptotic comparisons
to taking the mean of $2r-1$ independent draws.
\end{abstract}

\section{Introduction}

In this paper we study a median-of-means algorithm
for multidimensional randomized quasi-Monte Carlo (RQMC) sampling
over $[0,1]^s$ for $s\ge1$.  
The problem in RQMC is to estimate
$\mu =\int_{[0,1]^s}f(\bsx)\rd\bsx$.
The familiar Monte Carlo estimate is the
mean $\hat\mu$ of $f(\bsx_i)$ for $n$ independent
$\bsx_i\sim\dunif[0,1]^s$, with a root
mean squared error (RMSE) of $O(n^{-1/2})$
when $f$ has finite variance.
A quasi-Monte Carlo (QMC) estimate \cite{nied92} 
replaces those
$n$ points by deterministic points 
strategically chosen to more
uniformly sample the unit cube \cite{dick:pill:2010}.
The resulting absolute error is $O(n^{-1+\epsilon})$
for any $\epsilon>0$ when $f$ has finite total
variation in the sense of Hardy and Krause.
Randomizing those points \cite{rtms} 
in such a way that they remain digital nets
provides independent unbiased estimates of $\mu$
allowing one to estimate accuracy statistically.
For smooth enough $f$, the randomization also
improves the RMSE to $O(n^{-3/2+\epsilon})$
\cite{smoovar}.

The usual way to combine independent replicates of
randomized digital nets is to simply take the
average of the replicate estimates.  The method
we study here is to instead take the median
estimate from $2r-1$ independent replicates
when using the random linear 
scramble from \cite{mato:1998}.

In \cite{superpolyone} we studied the case $s=1$.
The median-of-means proposal in \cite{superpolyone}
uses a $(0,m,1)$-net in base $2$ randomized
with a random linear scramble of Matousek \cite{mato:1998}
and a digital shift.
For $f$ analytic on $[0,1]$ with integral $\mu$
estimated by an infinite precision RQMC
estimator denoted by $\hat\mu_\infty$ we saw
that the median of the randomization distribution of
$\hat\mu_\infty-\mu$ converges to $0$
as $O(n^{-c\log_2(n)})$ for any $c<3\log(2)/\pi^2\approx 0.21$.
That same rate could be attained by the sample median
of $2r-1$ independently replicated RQMC estimates so 
long as $r=\Omega(m)$ by
which we mean $m=O(r)$ as both $r$ and $m$ 
go to infinity.
That paper also considered integrands whose
$\alpha$ derivative satisfied a $\lambda$-H\"older
condition and found an error of
$O(n^{-\alpha-\lambda+\epsilon})$ for that case.
The significance of this result is that we can
attain a better rate than the customary
mean of replicated
estimates and that rate 
can adapt to an unknown smoothness 
level of the integrand without the user
having to know the smoothness level.  Indeed when
many integrals are computed from the same inputs
we might know that they have different smoothness
levels.

The previous paper was limited to $s=1$, where
there are many other good ways to integrate
a smooth function over $[0,1]$, as in \cite{davrab}.
That paper did however include a numerical 
result for 
the OTL circuit function on $[0,1]^6$ 
from \cite{surj:bing:2013}.
There the standard deviation of  a median
of means estimator was superior to that of the usual
mean-of-means at practically relevant sample sizes.
In the present paper we consider analytic functions
$f:[0,1]^s\to\real$.  We find that the median
value of $\hat\mu_\infty-\mu$ is now
$O(n^{-c\log_2(n)/s})$ for any $c<3\log(2)/\pi^2.$
In other words, there is still superlinear
convergence but with a dimension effect.

An outline of this paper is as follows.
Section~\ref{sec:background} introduces some
notation as well as the integration problem
and scrambling algorithms.
Section~\ref{sec:decomp} decomposes the
RQMC error into a sum over nonzero vectors
of $s$ nonnegative integers.  It is a sum 
of a randomly selected set of randomly 
signed Walsh coefficients.
That section introduces some notation that
we need to describe the complexity of the
Walsh basis functions and then presents
an upper bound on Walsh coefficients from Yoshiki \cite{yosh:2017}.
Section~\ref{sec:asymptotic}
gives asymptotic properties of the median
of means estimator. It bounds
the probability that a Walsh coefficient
contributes to the error and it shows
that the probability of an integration error
above $2^{-\lambda m^2/s+O(m\log(m))}$
is $O(1/\sqrt{m})$ when scrambling
a $(t,m,s)$-net in base $2$.
It also shows superpolynomial convergence
for some finite precision estimates
where the number of bits in the sample
values grows faster than a certain
multiple of $m^2/s$ and the median
of $\Omega(m^2)$ independent copies is used.
Section~\ref{sec:finitesample} 
looks at finite sample performance of the method
and gives conditions where a median-of-means
can outperform a mean-of-means for large $s$
and feasible $m$, despite the dimension
effect. This may happen when the
integrand is dominated by contributions from
a small set of important variables.
Section~\ref{sec:discussion} has a
discussion of the results focusing on
two remaining challenges: adaptation
to unknown smoothness, and quantifying uncertainty.
The median-of-means setting makes use of
techniques from analytic combinatorics that have
previously seen very little use in quasi-Monte
Carlo.  That literature has quite different
methods and notational conventions, and the
results we derive with it are in an Appendix.

We close the introduction with some bibliographic
remarks on median-of-means.  It is a longstanding
method in theoretical computer science. See
\cite{jerr:vali:vazi:1986} and \cite{lecu:lera:2020}
for some old and new uses, respectively.
Several uses in information based complexity are
discussed in \cite{kunsch2019solvable}.
Uses in quasi-Monte Carlo include \cite{superpolyone}
metioned above as well as 
\cite{hofstadler2022consistency} for some laws
of large numbers, \cite{goda:lecu:2022:tr}
for some smoothness adaptive lattice rules and
\cite{gobe:lera:meti:2022:tr} for robust RQMC
estimates.

\section{Notation and background}\label{sec:background}

We use $\natu$ for the natural numbers,
$\natu_0=\natu\cup\{0\}$
and $\ints_n=\{0,1,\dots,n-1\}$ for integers $n\ge2$. 
For $K\subset\natu$ we use $|K|$ for its cardinality.
We use $\caln=\{K \subset\natu\mid |K|<\infty\}$.
For $K\in\caln$ we use $\lceil K\rceil$ to
denote the largest element of $K$ with 
$\lceil\emptyset\rceil=0$ by convention.
When $x\in\real$ we use 
$\lceil x\rceil$ for the smallest integer greater than
or equal to $x$. The context will make it clear
whether the argument to $\lceil \cdot\rceil$ is
a real number or a set of natural numbers.

We let $\bszero$ be a vector of $m$ zeros
and we set $\natu^s_*=\natu_0^s\setminus\{\bszero\}$.
We abuse notation slightly by letting $\bszero$ be
either a row or a column vector as needed.
For $s\in \natu$ and $f\in L^2[0,1]^s$ we study approximation
of 
$$
\mu=\int_{[0,1]^s}f(\bsx)\rd\bsx
\quad
\text{by}
\quad
\hat\mu=\frac1n\sum_{i=0}^{n-1}f(\bsx_i)
$$
for $n\ge1$ and $\bsx_i\in[0,1]^d$.
We use $1{:}s=\{1,2,\dots,s\}$ for the set
of input indices to $f$.
When $v\subseteq1{:}s$ we use 
$-v$ for $1{:}s\setminus v$.




We use a van der Corput style mapping
between natural numbers and bit vectors and
points in $[0,1)$ as follows.
For $i\in\ints_{2^m}$ we let $\vec{i} = \vec{i}[m] = (i_1,i_2,\dots,i_m)^\tran\in\{0,1\}^m$
where $i = \sum_{\ell=1}^mi_\ell2^{\ell-1}$.
For $a=\sum_{\ell=1}^m a_\ell2^{-\ell}\in[0,1)$ we let
$\vec{a} = \vec{a}[E] = (a_1,a_2,\dots,a_E)^\tran$.
Here $E$ is the precision of $\vec{a}$
and we typically have $E\ge m$ in our use cases. 
For $a$ with two binary expansions we choose the one ending
in infinitely many 0s. 
For each $\vec{a}$ there is a unique 
$a\in[0,1)$. When $E<\infty$, we can
have $\vec{a}=\vec{a\,}'$ for $a\ne a'$.

For an integer base $b\ge2$ and vectors $\bsk,\bsc\in\natu_0^s$ with $0\le c_j<b^{k_j}$,
an elementary interval in base $b$ is a Cartesian
product of the form
$$
\prod_{j=1}^s\Bigl[
\frac{c_j}{b^{k_j}},\frac{c_j+1}{b^{k_j}}\Bigr).
$$
For integers $m\ge t\ge0$, the points $\bsx_0,\dots,\bsx_{b^m-1}\in[0,1)^s$
form a $(t,m,s)$-net in base $b$ if every
elementary interval with $\sum_{j=1}^sk_j=m-t$
contains precisely $b^t$ of those points.
Here, $t$ is the quality parameter of the net
with smaller values being better.  It is not
always possible to get $t=0$ for a given
choice of $b$ and $m$ and $s$.
The infinite sequence $\bsx_i\in[0,1)^s$ for $i\in\natu_0$ forms a $(t,s)$-sequence in base
$b$ if for all integers $m\ge t$ and $r\ge0$, 
the points $\bsx_{rb^m},\dots,\bsx_{(r+1)b^m-1}$
form a $(t,m,s)$-net in base $b$.
In this paper we consider $b=2$. This includes
the most widely used nets of Sobol' \cite{sobo:1967}
as well
as those of Niederreiter and Xing \cite{niedxing96} that have
some of the best available $t$ values.

Base $2$ digital nets of $n=2^m$ points are formed by setting
\begin{align}\label{eq:defveca}
\vec{a}_{ij} = C_j\vec{i}\ \tmod 2
\end{align}
for $0\le i<2^m$ and $j=1,\dots,s$
for carefully chosen generator matrices 
$C_j =C_j[E]\in\{0,1\}^{E\times m}$
where $E\ge m$ is a precision.
Our theoretical
analysis emphasizes $E=\infty$.
The attained value of $t$ is a property
of the chosen generator matrices.
We always assume that $C_j$ has full rank
over $\ints_2$.
The points $\bsa_i\in[0,1)^s$ have components
$a_{ij}$ determined by $\vec{a}_{ij}$ from equation~\eqref{eq:defveca}.  That is, we give
expressions for $\vec{a}_{ij}$ with the
understanding that $a_{ij}\in[0,1)
=\sum_{\ell=1}^E2^{-\ell}a_{ij\ell}$
when $\vec{a}_{ij} = (a_{ij1},a_{ij2},\dots,a_{ijE})$.

For a base $2$ digital $(t,s)$-sequence one uses generator
matrices with infinitely many rows and columns.
For $m\ge t$, the first $n=2^m$ points of such a sequence
are a $(t,m,s)$-net in base $2$.
When we consider a digital sequence we suppose
that for each finite $m$ we are working with
$C_j$ equal to the upper left $E\times m$ submatrix
of the infinite generator matrix. Note that any entries
in $\vec{i}$ after the $m$'th are zero for $i\in\ints_{2^m}$
so columns of $C_j$ after the $m$'th do not affect $\hat\mu_n$.

Given points $\bsa_i = (a_{i1},\dots,a_{is})\in[0,1)^s$ 
of a digital net, we
define linearly scrambled points as follows.
For precision $E\ge m$ we choose random matrices $M_j\in\{0,1\}^{E\times m}$
and random vectors $D_j \in\{0,1\}^{E}$
and take
\begin{equation}\label{eqn:xequalMCiplusD}
    \vec{x}_{ij}=\vec{x}_{ij}[E]=\vec{a}_{ij}+\vec{D}_j=M_jC_j\vec{i}+\vec{D}_j
\ \tmod 2
\end{equation}
for $0\le i<2^m$ and $1\le j\le s$ to define $\bsx_i\in[0,1]^s$.
From here on, arithmetic operations on bit vectors are 
taken modulo two unless otherwise indicated.
Our estimate of $\mu$ is now
$$
\hat\mu = \hat\mu_E = \frac1n\sum_{i=0}^{n-1}f(\bsx_i).
$$
For $E'<E$ we define $\hat\mu_{E'}$ as above
keeping only the first $E'$ rows of $M_j$ and
the first $E'$ entries in $D_j.$ Our reduced
precision estimate $\hat\mu_{E'}$ uses
$\vec{x}_{ij}[E']=M_j(1{:}E',:)C_j\vec{i}+\vec{D}_j(1{:}E')$.

\begin{lemma}\label{lem:precisionerror}
Let $f:[0,1]^s\to\real$ have modulus of continuity $\omega_f$.
Let $M_j$ and $D_j$
for $j=1,\dots,s$ be defined with 
infinite precision. Then
$$|\hat\mu_\infty-\hat\mu_E|\leq \omega_f\Bigl(\frac{\sqrt{s}}{2^E}\Bigr).$$
\end{lemma}
\begin{proof} Let $\bsx_i[E]$ be $\bsx_i$ under scrambling with
  precision $E$ and $\bsx_i[\infty]$ be $\bsx_i$ under scrambling in
  the infinite precision limit. By Lemma 1 of \cite{superpolyone}, each coordinate of $\bsx_i[E]$ differs from $\bsx_i[\infty]$ by at most $2^{-E}$. Therefore 
$\Vert\bsx_i[E]-\bsx_i[\infty]\Vert_2\leq \sqrt{s}2^{-E}$ and so
$$|\mu_\infty-\mu_E|\leq \frac{1}{n}\sum_{i=0}^{n-1} |f(\bsx_i[E])-f(\bsx_i[\infty])|\leq \omega_f\Bigl(\frac{\sqrt{s}}{2^E}\Bigr).
\qedhere$$
\end{proof}

We will use $\omega_f(\sqrt{s})$ as 
shorthand for
$\sup_{\bsx\in [0,1]^s}f(\bsx)-\inf_{\bsx\in[0,1]^s}f(\bsx)$.

We focus on the random linear scrambling of \cite{mato:1998}.
The matrix $M_j\in\{0,1\}^{E\times s}$
is lower triangular with ones on the diagonal
and independent $\dunif\{0,1\}$ entries
below the diagonal.  The digital shift
has independent $\dunif\{0,1\}$ elements.
That is
$$M_{j,\ell\ell'}
=\begin{cases}
0, & 1\le \ell <\ell'\le m\\
1, & 1\le \ell =\ell'\le m\\
\dunif\{0,1\}, &\text{else,}
\end{cases}
$$
and $D_{j,\ell}=\dunif\{0,1\}$
for $\ell=1,\dots,E$.
We sketch this setting for $m=3$ and $E=4$
as follows:
\begin{align}\label{eq:thesketch}
M_j = 
\begin{pmatrix}
1 &\\
u & 1 \\
u & u & 1\\
u & u & u 
\end{pmatrix}
\quad\text{and}\quad D_j = \begin{pmatrix}
u\\
u\\
u\\
u
\end{pmatrix},
\end{align}
with $u$ representing random elements.
All of the uniform random variables in
$M_1,\dots,M_s$ and $D_1,\dots,D_s$ are
independent.

\section{Error decomposition}\label{sec:decomp}

In order to analyze the convergence rate of median-of-means, we first derive an error decomposition formula for $\hat{\mu}_{\infty}-\mu$ using Walsh functions.
For $k\in \natu_0$ and $x\in [0,1)$, we define 
\begin{align}\label{eq:defwalk}
\walk(x)=(-1)^{\vec{k{}}^\tran \vec{x}}.
\end{align}
Because $k$ is a finite integer, only finitely many
entries in $\vec{k}$ are nonzero and so the inner product
in~\eqref{eq:defwalk} is a finite
sum. For the multivariate generalization, the $\bsk$'th dyadic Walsh
function $\walbsk(\bsx)$ for $\bsk\in\natu_0^s$ is defined to be
\begin{equation}\label{eqn:Walshfunction}
    \walbsk(\bsx)=\prod_{j=1}^s \mathrm{wal}_{k_j}(x_j)=(-1)^{\sum_{j=1}^s\vec{k{}}_j^{\tran} \vec{x}_j}.
\end{equation}

It is known that $\{\walbsk(\bsx)\mid \bsk\in \natu_0^s\}$ form a
complete orthonormal basis of $L^2[0,1)^s$
\cite{dick:pill:2010}. Therefore for $f\in L^2[0,1)^s$
\begin{align}\label{eqn:Walshdecomposition}
   f(\bsx) &=\sum_{\bsk\in \natu_0^s}\hat{f}(\bsk) \walbsk(\bsx),
\quad\text{where}\\
     \hat{f}(\bsk)&=\int_{[0,1)^s}f(\bsx)\walbsk(\bsx)\rd\bsx.
\label{eqn:Walshcoefficient}
\end{align}
Equation~\eqref{eqn:Walshdecomposition} holds in a mean square sense.

\begin{theorem}\label{thm:decomp}
Let $f\in L^2[0,1)^s$ and let $\bsx_i$ be defined
by~\eqref{eqn:xequalMCiplusD} for $0\le i<2^m$.
Then
\begin{equation}\label{eqn:errordecomposition}
    \hat{\mu}_{\infty}-\mu=\sum_{\bsk\in \natu_*^s}\bsone\Biggl\{\sum_{j=1}^s \vec{k{}}_j^\tran M_j C_j=\bszero\Biggr\}\hat{f}(\bsk) (-1)^{\sum_{j=1}^s \vec{k{}}_j^\tran \vec{D}_{j}}.
\end{equation}
\end{theorem}
\begin{proof} From equation~\eqref{eqn:Walshcoefficient} we see that $\mu=\hat{f}(\bszero)$. So by equations~\eqref{eqn:Walshfunction} and \eqref{eqn:Walshdecomposition},
$$\hat{\mu}_{\infty}-\mu=\sum_{\bsk\in \natu_*^s}\hat{f}(\bsk)\frac{1}{n}\sum_{i=0}^{n-1} \walbsk(\bsx_i)=\sum_{\bsk\in \natu_*^s}\hat{f}(\bsk)\frac{1}{n}\sum_{i=0}^{n-1} (-1)^{\sum_{j=1}^s\vec{k{}}_{j}^\tran \vec{x}_{ij}}.$$
From equation~\eqref{eqn:xequalMCiplusD}, we have
\begin{align*}
    \frac{1}{n}\sum_{i=0}^{n-1} (-1)^{\sum_{j=1}^s\vec{k{}}_{j}^\tran \vec{x}_{ij}}&=\frac{1}{n}\sum_{i=0}^{n-1} (-1)^{\sum_{j=1}^s\vec{k{}}_{j}^\tran (M_j C_j\vec{i}+\vec{D}_j)}\\
    &=(-1)^{\sum_{j=1}^s \vec{k{}}_j^\tran \vec{D}_{j}} \frac{1}{n}\sum_{i=0}^{n-1} (-1)^{\sum_{j=1}^s\vec{k{}}_{j}^\tran M_j C_j\vec{i}}\\
    &=(-1)^{\sum_{j=1}^s \vec{k{}}_j^\tran \vec{D}_{j}}\bsone\Bigl\{\sum_{j=1}^s \vec{k{}}_j^\tran M_j C_j=\bszero\Bigl\}
\end{align*}
so the conclusion follows.
\end{proof}

We need to quantify several
properties of Walsh function indices
$k$ and $\bsk$. 
Let $k\in\natu_0$ have binary expansion $k=\sum_{\ell=1}^\infty b_\ell 2^{\ell-1}$
for bits $b_\ell\in\{0,1\}$.
First we let
\begin{align}\vec{k} &:= \vec{k}[\infty] = (b_1,b_2,\dots)^\tran,\quad\text{and}\\
\kappa & := \{ \ell\in\natu \mid b_\ell=1\}.
\end{align}
We will study $\walk$ using the cardinality of $\kappa$, the
sum of its elements, and its last (largest) element.
For $k\ge1$, these are
$$
|\kappa|,\quad
\Vert\kappa\Vert_1=\sum_{\ell\in\kappa}\ell,
\quad\text{and}\quad
\lceil \kappa\rceil = \max_{\ell\in\kappa}\ell,
$$
respectively.  
For $k=0$ we set
$\kappa=\emptyset$ and then
$\Vert\kappa\Vert_0=\Vert\kappa\Vert_1=\lceil\kappa\rceil=0$,
the last one by convention.  


In the $s$ dimensional setting we need to vectorize
these quantities.
For $\bsk=(k_1,\dots,k_s)\in\natu^s$, we define
the corresponding vectors $\vec{k}_1,\dots,\vec{k}_s$
and sets $\kappa_1,\dots,\kappa_s$ componentwise.
We need to keep track of those indices in
$\bsk$ for which $k_j>0$.  We denote the 
supports of $\bsk$ and $\bskappa$ as
$\supp(\bsk)=\{j\in 1{:}s\mid k_j>0\}$
and $\supp(\bskappa)
=\{j\in 1{:}s\mid \kappa_j\ne\emptyset\}$
respectively. Clearly $\supp(\bsk)=\supp(\bskappa)$.

Given $\bsk\in\natu_0^s$ we now define
the corresponding
bit matrix $\vec{\bsk}=(\vec{k}_1,\dots,\vec{k}_s)$
along with 
$\bskappa=(\kappa_1,\dots,\kappa_s)\in\caln^s$,
a list of finite sets of natural numbers.
We need some componentwise quantities
for $\bskappa$ and some aggregate quantities.
The componentwise quantities are
$$
\lceil\bskappa\rceil 
= (\lceil\kappa_1\rceil,\dots,
\lceil\kappa_s\rceil)\in\natu_0^s
\quad\text{and}\quad
|\bskappa| = (|\kappa_1|,\dots,|\kappa_s|)
\in\natu_0^s.
$$
The first two aggregate quantities are
$$
\Vert\bskappa\Vert_1 = \sum_{j=1}^s\Vert\kappa_j\Vert_1
\quad\text{and}\quad
\Vert\bskappa\Vert_0=\sum_{j=1}^s|\kappa_j|.
$$
Note that $\Vert\bskappa\Vert_0$ is the
number of one bits in $\vec{\bsk}$.
We also need the sum of largest indices
$$
\Vert\lceil\bskappa\rceil\Vert_1 = \sum_{j=1}^s\lceil\kappa_j\rceil.
$$
These quantities satisfy
$$
\Vert\bskappa\Vert_0 \le \Vert \lfloor\bskappa\rfloor\Vert_1 \le \Vert \bskappa\Vert_1.
$$

Theorem 2 of \cite{yosh:2017}
provides the following crucial bound on $|\hat{f}(\bsk)|$.
\begin{lemma}\label{lem:Walshcoefficientbound}
Let $f\in C^{\infty}[0,1)^s$. Then
\begin{align*}|\hat{f}(\bsk)|&\leq 2^{-\Vert\bskappa\Vert_1-\Vert\bskappa\Vert_0}\sup_{\bsx_{\supp(\bsk)}\in[0,1)^{|\supp(\bsk)|}}\Big|\int_{[0,1)^{s-|\supp(\bsk)|}}f^{|\bskappa|}(\bsx)\rd\bsx_{-\supp(\bsk)}\Big|
\end{align*}
where
\begin{align*}
f^{|\bskappa|}&=f^{(|\kappa_1|,\dots,|\kappa_s|)}=\frac{\partial^{\Vert\bskappa\Vert_0}f}{\partial x_1^{|\kappa_1|}\cdots\partial x_s^{|\kappa_s|}}.
\end{align*}
\end{lemma}

Yoshiki's Theorem 2 uses a norm defined in
his Theorem 1 for smoothness $\alpha\ge2$.
Our setting has $\alpha=\infty$.
We take his $p=\infty$.
Our $\Vert\bskappa\Vert_0+\Vert\bskappa\Vert_1$
is his $\mu'_\alpha(\bsk_v)$.




\section{Asymptotic convergence rate}\label{sec:asymptotic}

In this section we derive the super-polynomial convergence rate of median-of-means. Many parts of the analysis will be refined in the next section to derive a tighter finite sample bound.

As a first step, we want to know the probability that $\sum_{j=1}^s
\vec{k{}}_j^\tran M_j C_j=\bszero$ when $M_j$ is generated by random
linear scrambling. Recall that we have assumed that each $C_j$ is nonsingular. 
We let $C_j(1{:}q,:)$ denote the first $q\ge0$ rows
of $C_j$ and then for $q_1,\dots,q_s\in\natu$ we write
$$C^{(q_1,\dots,q_s)}=\begin{bmatrix}
C_1(1{:}q_1,:)\\
\vdots\\
C_s(1{:}q_s,:)
\end{bmatrix}
\in\{0,1\}^{(\sum_{j=1}^sq_j)\times m}$$
with the convention that when $q_j=0$, $C_j(1{:}q_j,:)$ is an empty matrix.
If $(q_1,\dots,q_s)=\bszero$, we define $C^{(q_1,\dots,q_s)}$ to be a
$0\times m$ matrix and it has rank 0. We will use $\row(C)$ to denote the row
space of matrix $C$ in  $\{0,1\}^m$. 
For $\bsv\subseteq1{:}s$
we let $\bsone\{\bsv\}\in\{0,1\}^s$
be the vector with $v_j=1$ for $j\in v$ and $v_j=0$ for $j\not\in v$.

A very important quantity that recurs in our
analysis is the matrix
$C^{\lceil\bskappa\rceil}$.
For every $j$ with $k_j>0$, this
matrix has all the rows of $C_j$
that will be relevant to $\walbsk(\bsx_i)$, namely $C_j(1{:}\lceil\kappa_j\rceil,:)$.  If we remove the last relevant
row of each $C_j$ we obtain
$C^{\lceil\bskappa\rceil-\bsone\{\supp(\bskappa)\}}$.


\begin{lemma} \label{lem:gaincoef}
If $\max_{1\le j\le s}\lceil\kappa_j\rceil>m$, then
$$\Pr\bigg(\sum_{j=1}^s \vec{k{}}_j^\tran M_j C_j=\bszero\bigg)=2^{-m}.$$
If $\max_{1\le j\le s}\lceil\kappa_j\rceil\leq m$ and $\sum_{j\in
  \supp(\bsk)}C_j(\lceil\kappa_j\rceil,:)\in \row(C^{\lceil\bskappa\rceil-\bsone\{\supp(\bsk)\}})$,  then
$$\Pr\bigg(\sum_{j=1}^s \vec{k{}}_j^\tran M_j C_j=\bszero\bigg)=2^{-\rank(C^{\lceil\bskappa\rceil-\bsone\{\supp(\bsk)\}})}.$$
Otherwise 
$$\Pr\bigg(\sum_{j=1}^s \vec{k{}}_j^\tran M_j C_j=\bszero\bigg)=0.$$
\end{lemma}
\begin{proof}
Because of the upper triangular form for $M_j$
(recall the sketch in equation~\eqref{eq:thesketch}),
we see that $\vec{k{}}_j^\tran M_j$ has the same distribution as $M_j(\lceil\kappa_j\rceil,:)$.
    Because $C_j$ is nonsingular, if $\lceil k_{j^*}\rceil> m$ for any $j^*\in 1{:}s$, then $\vec{k{}}_{j^*}^\tran M_{j^*} C_{j^*}$ is uniformly distributed on the set of $2^m$ possible binary vectors so that
\begin{align*}
Pr\Big(\sum_{j=1}^s \vec{k{}}_j^\tran M_j C_j=\bszero\Big)
&=\Pr\Big(\vec{k{}}_{j^*}^\tran M_{j^*} C_{j^*}+\sum_{j\in 1{:}s, j\neq j^*} \vec{k{}}_j^\tran M_j C_j=\bszero\Big)\\
&=\Pr\Big(\vec{k{}}_{j^*}^\tran M_{j^*} C_{j^*}=\sum_{j\in 1{:}s, j\neq j^*} \vec{k{}}_j^\tran M_j C_j\Big)=2^{-m}
\end{align*}
establishing the first claim.

Now assume that all $\lceil\kappa_j\rceil\leq m$. Then
$$\sum_{j=1}^s \vec{k{}}_j^\tran M_j C_j
=\sum_{j\in \supp(\bsk)}C_j(\lceil\kappa_j\rceil,:)+\sum_{j\in \supp(\bsk)}
\Bigl(\vec{k{}}_j^\tran M_jC_j-C_j(\lceil\kappa_j\rceil,:)\Bigr).$$
Observe that $\vec{k{}}_j^\tran M_jC_j-C_j(\lceil\kappa_j\rceil,:)$ is uniformly distributed on the linear span of first $\lceil\kappa_j\rceil-1$ rows of $C_j$. Hence the second sum on the right is uniformly distributed on $\row(C^{\lceil\bskappa\rceil-\bsone\{\supp(\bsk)\}})$. If $\sum_{j\in \supp(\bsk)}C_j(\lceil\kappa_j\rceil,:)\in \row(C^{\lceil\bskappa\rceil-\bsone\{\supp(\bsk)\}})$, then
$$\Pr\Big(\sum_{j=1}^s \vec{k{}}_j^\tran M_j C_j=\bszero\Big)=\frac{1}{|\row(C^{\lceil\bskappa\rceil-\bsone\{\supp(\bsk)\}})|}=2^{-\rank(C^{\lceil\bskappa\rceil-\bsone\{\supp(\bsk)\}})}$$
establishing the second claim.
If $\sum_{j\in \supp(\bsk)}C_j(\lceil\kappa_j\rceil,:)\notin \row(C^{\lceil\bskappa\rceil-\bsone\{\supp(\bsk)\}})$, then the above probability is clearly 0,
establishing the final claim.
\end{proof}

\begin{corollary}\label{cor:tmsbound}
If $C_1,\dots,C_s$ generate a digital $(t,m,s)$-digital in base $2$, then 
$$\Pr\Big(\sum_{j=1}^s \vec{k{}}_j^\tran M_j C_j=\bszero\Big)\leq 2^{-m+t+s}.$$
\end{corollary}
\begin{proof}
    We only need to verify that
    $\rank(C^{\lceil\bskappa\rceil-\bsone\{\supp(\bsk)\}})\geq m-t-s$ when
    $\max_{1\le j\le s} \lceil\kappa_j\rceil\leq m$ and $\sum_{j\in \supp(\bsk)}C_j(\lceil\kappa_j\rceil,:)\in \row(C^{\lceil\bskappa\rceil-\bsone\{\supp(\bsk)\}})$. Notice that in this case $C^{\lceil\bskappa\rceil}$ is rank-deficient. By the definition of $(t,m,s)$-digital net, a rank-deficient $C^{\lceil\bskappa\rceil}$ must contains $m-t$ linearly independent rows, so $\rank(C^{\lceil\bskappa\rceil})\geq m-t$. Hence
    $$\rank(C^{\lceil\bskappa\rceil-\bsone\{\supp(\bsk)\}})\geq \rank(C^{\lceil\bskappa\rceil})-|\supp(\bsk)|\geq m-t-s$$
    which proves the conclusion.
\end{proof}

\begin{corollary}\label{cor:concentration}
Let $\lambda=3\log(2)^2/\pi^2\approx 0.146$. 
For $j=1,\dots,s$ let $C_j=C_j(m)$
be the first $m$ columns of the
generator matrices of
a digital $(t,s)$-net in base $2$.
Then
$$\Pr\bigg(\sum_{j=1}^s \vec{k{}}_j^\tran M_j C_j=\bszero \text{ for some } \bsk\ne\bszero \text{ with } 
\Vert \bsk\Vert_1\leq 
\frac{\lambda m^2}{s}\bigg)=O\Bigl(\frac{1}{\sqrt{m}}\Bigr)$$
as $m\to \infty$.
\end{corollary}
\begin{proof}
    From Corollary~\ref{cor:growthrate} in the appendix, we know that 
   $$ \bigl|\{\bsk\in \natu_*^s\mid \Vert\bskappa\Vert_1\leq {\lambda m^2}/{s}\}\bigr|=\Theta\Bigl(\frac{2^m}{\sqrt{m}}\Bigr).$$
   So from the union bound on 
   the result of Corollary~\ref{cor:tmsbound}
    \begin{align*}
        &\Pr\bigg(\,\sum_{j=1}^s \vec{k{}}_j^\tran M_j C_j=\bszero
          \text{ for some } \bsk
          \in\natu_*^s \text{ with } \Vert\bskappa\Vert_1\leq \frac{\lambda m^2}{s}\bigg)\\
        &\leq 2^{-m+t+s}\bigl|\{\bsk\in \natu_*^s\mid \Vert\bskappa\Vert_1\leq {\lambda m^2}/{s}\}\bigr|\\
        &=O\Bigl(\frac{1}{\sqrt{m}}\Bigr). \qedhere
    \end{align*}
\end{proof}

Now we are ready to prove the main theorem that shows $|\hat{\mu}_{\infty}-\mu|=2^{-\lambda m^2/s+O(m\log m)}$ with high probability.
We note that for $f$ to be analytic over $[0,1]^s$ means
that it equals its infinite order Taylor expansion
on some open set containing $[0,1]^s$.

\begin{theorem} \label{thm:superconvergence}
Let $f$ be analytic over $[0,1]^s$.
Let $\bsx_i$ be from a $(t,s)$-sequence in
base $2$ with a random linear scramble
plus digital shift.
Then there exist constants $B_1$ and $B_2$ such that for all $m\geq 2$ 
$$\Pr\Bigl(|\hat{\mu}_{\infty}-\mu|\geq 2^{-\lambda m^2/s+B_1m\log m}\Bigr)\leq \frac{B_2}{\sqrt{m}}.$$

\end{theorem}
\begin{proof}
    Because $[0,1]^s$ is compact, we can find $\epsilon>0$ such that for all $\bst\in [0,1]^s$, the Taylor expansion of $f$ centered at $\bst$
        $$\sum_{n_1=0}^\infty\dots\sum_{n_s=0}^\infty \frac{\partial^{n_1+\dots+n_s}f}{\partial x_1^{n_1}\dots\partial x_s^{n_s}}(\bst)\prod_{j=1}^s \frac{(x_j-t_j)^{n_j}}{n_j!}$$
    converges absolutely in an edge-length-$2\epsilon$ box centered at $\bst$. 
It follows that
    $$\bigg|\frac{\partial^{n_1+\dots+n_s}f}{\partial x_1^{n_1}\dots\partial x_s^{n_s}}(\bst)\bigg|\prod_{j=1}^s \frac{\epsilon^{n_j}}{n_j!}\to 0$$
    as $n_1+\dots+n_s\to \infty$.
    There must then be a constant $A$ such that
    $$\Big|\frac{\partial^{n_1+\dots+n_s}f}{\partial x_1^{n_1}\dots\partial x_s^{n_s}}(\bst)\Big|\leq \frac{A n!}{\epsilon^n}$$
    holds for all $\bst\in[0,1]^s$
    where $n=n_1+\dots+n_s$. Lemma~\ref{lem:Walshcoefficientbound}
    then implies that
\begin{align*}|\hat{f}(\bsk)|&\leq 2^{-\Vert\bskappa\Vert_1-\Vert\bskappa\Vert_0}\sup_{\bst\in [0,1]^s}\Big|\frac{\partial^{\Vert\bskappa\Vert_0}f}{\partial x_1^{|\kappa_1|}\cdots\partial x_s^{|\kappa_s|}}(\bst)\Big|\\
&\leq  A 2^{-\Vert\bskappa\Vert_1} \Bigl(\frac{1}{2\epsilon}\Bigr)^{\Vert\bskappa\Vert_0} \Vert\bskappa\Vert_0!.
\end{align*}
Because $\epsilon$ can be chosen arbitrarily small, we assume without loss of generality that $2\epsilon<1$.

Let $\ce$ be the event that
no $\bsk\in\natu^s_*$ with $\Vert\bskappa\Vert_1\leq \lambda m^2/s$ has
$\sum_{j=1}^s \vec{k{}}_j^\tran M_j C_j=\bszero$. 
Corollary~\ref{cor:concentration} shows that $\Pr(\ce) = 1-O(1/\sqrt{m})$ and 
we take $B_2$ to be the implied constant
in that expression.
Conditionally on $\ce$, equation~\eqref{eqn:errordecomposition} becomes
    \begin{align*}
        |\hat{\mu}_{\infty}-\mu|&=\Biggl|
\sum_{\bsk\in \natu_*^s:\Vert\bskappa\Vert_1>\lambda m^2/s}
\bsone\biggl\{\sum_{j=1}^s \vec{k{}}_j^\tran M_j C_j=\bszero\biggr\}
\hat{f}(\bsk) (-1)^{\sum_{j=1}^s \vec{k{}}_j^\tran \vec{D}_{j}}\Biggr| \nonumber\\ 
        &\leq \sum_{\bsk\in \natu_*^s:\Vert\bskappa\Vert_1>\lambda m^2/s}|\hat{f}(\bsk)| \nonumber\\ 
        &\leq A\times\sum_{\bsk\in \natu_*^s:\Vert\bskappa\Vert_1>\lambda m^2/s} 2^{-\Vert\bskappa\Vert_1} 
\Vert\bskappa\Vert_0!/(2\epsilon)^{\Vert\bskappa\Vert_0} .
    \end{align*}
    Now for $k=\sum_{\ell=1}^\infty b_\ell 2^{\ell-1}=\sum_{\ell\in\kappa}2^{\ell-1}$
\begin{align*}
\Vert\kappa\Vert_1&=\sum_{\ell\in\kappa} \ell \geq\sum_{\ell=1}^{|\kappa|}\ell
\ge\frac{|\kappa|^2}{2},
\end{align*}
with equality holding for $\kappa=\emptyset$. Then
 \begin{align*}
 \Vert\bskappa\Vert_1&=\sum_{j=1}^s \Vert\kappa_j\Vert_1
    \ge
  \sum_{j=1}^s\frac{|\kappa_j|^2}{2}\geq
  \frac{1}{2s}\biggl(\sum_{j=1}^s|\kappa_j|\biggr)^2=\frac{1}{2s}\Vert\bskappa\Vert_0^2,
\end{align*}
yielding $\Vert\bskappa\Vert_0\leq \sqrt{2s \Vert\bskappa\Vert_1}$. Hence
    \begin{align*}
        &\sum_{\bsk\in \natu_*^s:\Vert\bskappa\Vert_1>\lambda m^2/s}A 2^{-\Vert\bskappa\Vert_1} \Vert\bskappa\Vert_0!/(2\epsilon)^{\Vert\bskappa\Vert_0} \\
        &\leq \sum_{N=\lceil \lambda m^2/s\rceil }^\infty A 2^{-N}\Bigl(\frac{1}{2\epsilon}\Bigr)^{\sqrt{2s N}}\Gamma(\sqrt{2sN}+1)|\{\bsk\in \natu_*^s\mid \Vert\bskappa\Vert_1=N\}|
    \end{align*}
where $\Gamma(\cdot)$ is the Gamma function
and we have also used $2\epsilon<1$.

By Theorem~\ref{thm:finiteNbound} in the appendix, 
    $$|\{\bsk\in \natu_*^s\mid \Vert\bskappa\Vert_1=N\}|\leq \frac{B}{\sqrt{N}}\exp\Bigl(\pi\sqrt{\frac{sN}{3}}\Bigr)$$
holds for $B=\pi\sqrt{s}/(2\sqrt{3})$. Hence
$$\Bigl(\frac{1}{2\epsilon}\Bigr)^{\sqrt{2s N}}\Gamma(\sqrt{2sN}+1)|\{\bsk\in \natu_*^s\mid \Vert\bskappa\Vert_1=N\}|\leq 2^{\newconstant\sqrt{N}\log (N)}$$
for some constant $\newconstant$. Because $\sqrt{N+1}\log (N+1)-\sqrt{N}\log (N)$ converges to $0$ as $N\to \infty$, we can find $N_\rho$ for any $\rho>1$ such that $2^{\newconstant\sqrt{N}\log (N)}<\rho^{N}$ for $N>N_\rho$. Let us choose $\rho=3/2$ for simplicity. Then when $\lambda m^2/s>N_{3/2}$, 
\begin{align*}
    \sum_{N=\lceil \lambda m^2/s\rceil }^\infty  2^{-N}2^{\newconstant\sqrt{N}\log (N)}&\leq 2^{-\lceil \lambda m^2/s\rceil}2^{\newconstant\sqrt{\lceil \lambda m^2/s\rceil}\log (\lceil \lambda m^2/s\rceil)}\sum_{N=0}^\infty \Bigl(\frac{3}{4}\Bigr)^{N}\\
    &\leq 2^{-\lambda m^2/s+B_1 m\log(m)}
\end{align*}
for some constant $B_1$. The conclusion follows once we increase $B_1$ sufficiently to cover all $m\geq 2$ cases.


\end{proof}

\begin{corollary}\label{cor:medianMSE}
Under the same condition as Theorem~\ref{thm:superconvergence}, if $E\geq \lambda m^2/s$ and $r=\Omega(m^2)$, then the sample median $\hat{\mu}^{(r)}_{E}$ of  $2r-1$ independently generated values of $\hat{\mu}_E$ satisfies
$$\e(|\hat{\mu}^{(r)}_{E}-\mu|^2)\leq 4^{-\lambda m^2/s+O(m\log(m))}.$$
\end{corollary}
\begin{proof} By Lemma~\ref{lem:precisionerror}, with probability at least $1-{B_2}/{\sqrt{m}}$, 
$$|\hat{\mu}_E-\mu|\leq  2^{-\lambda m^2/s+B_1m\log m}+\frac{\sqrt{s}}{2^E}
\sup_{\bsx\in [0,1]^s}\Vert\nabla f(\bsx)\Vert_2 $$
where $B_1$ and $B_2$ come from Theorem~\ref{thm:superconvergence} and $\nabla f$ is the gradient of $f$. 

In order for the sample median of $2r-1$ copies of $\hat{\mu}_E$ to violate the 
above bound, there must be at least $r$ copies violating the bound. Because 
there are ${2r-1\choose r}$ subsets of size $r$, the union bound implies 
that the probability of $r$ such violations is at most
$${2r-1\choose r}\Bigl(\frac{B_2}{\sqrt{m}}\Bigr)^{r}=\prod_{j=2}^r \frac{(2j-1)(2j-2)}{j(j-1)} \Bigl(\frac{B_2}{\sqrt{m}}\Bigr)^{r}< \Bigl(\frac{4B_2}{\sqrt{m}}\Bigr)^{r}.$$
When the above described event happens, $|\hat{\mu}^{(r)}_{E}-\mu|$ is still 
bounded by $\sup_{\bsx\in[0,1]^s}|f(\bsx)|$. Hence
\begin{align*}
\e(|\hat{\mu}^{(r)}_{E}-\mu|^2)&\leq 
\Bigl(2^{-\lambda m^2/s+B_1m\log m}+O\Bigl(\frac{1}{2^E}\Bigr)\Bigr)^2+O\Bigl(\Bigl(\frac{4B_2}{\sqrt{m}}\Bigr)^{r}\Bigr)\\
&=4^{-\lambda m^2/s+O(m\log(m))}
\end{align*}
under our assumptions on $E$ and $r$.
\end{proof}

\section{Finite sample analysis}\label{sec:finitesample}
Although the asymptotic convergence rate of median-of-means is
super-poly\-nomial, the bound in Corollary~\ref{cor:medianMSE} is of
limited use when $\lambda m^2/s$ is only moderately large or even
smaller than $m$. In this section, we derive results that
better describe the finite sample behavior of median-of-means. In
particular, we want to study under what conditions median-of-means can
outperform the usual RQMC estimator (mean-of-means) in terms of mean
squared error. For simplicity, we assume that the precision $E$ is high enough that the difference between $\hat{\mu}_\infty$ and $\hat{\mu}_E$ is negligible in comparison to their root mean squared error.

First let us work out the variance of $\hat{\mu}_\infty$.

\begin{lemma}\label{lem:uncorrelated}
  For $\bsk\in \natu^s_*$,
  let $S(\bsk) = (-1)^{\sum_{j=1}^s \vec{k{}}_j^\tran \vec{D}_{j}}$. Then
$$\Pr(S(\bsk)=1)=\Pr(S(\bsk)=-1)=1/2.$$
For distinct $\bsk,\bsk'\in\natu^s_*$, $S(\bsk)$ and $S(\bsk')$
are independent.
\end{lemma}
\begin{proof}
The proof is similar to Lemma 4 of \cite{superpolyone} and is omitted here.
\end{proof}

\begin{theorem} $\e(\hat{\mu}_\infty)=\mu$ and
\begin{equation*}
    \var(\hat{\mu}_\infty)=\sum_{\bsk\in \natu_*^s}
    \Pr\Biggl(\,\sum_{j=1}^s \vec{k{}}_j^\tran M_j C_j=\bszero\Biggr)\hat{f}(\bsk)^2.
\end{equation*}
\end{theorem}
\begin{proof}
Let $\bsM=(M_1,M_2,\dots,M_s)$.
By equation~\eqref{eqn:errordecomposition} and 
Lemma~\ref{lem:uncorrelated}, 
\begin{equation}\label{eqn:conditionallyunbiased}
    \e(\hat{\mu}_\infty-\mu\giv \bsM)=\sum_{\bsk\in \natu_*^s}\bsone\Bigl\{\sum_{j=1}^s \vec{k{}}_j^\tran M_j C_j=\bszero\Bigl\}\hat{f}(\bsk) \e(S(\bsk))=0
\end{equation}
and
\begin{align}\label{eqn:conditionalvariance}
    \var(\hat{\mu}_\infty-\mu\giv \bsM)&=\sum_{\bsk\in \natu_*^s}\bsone\Bigl\{\sum_{j=1}^s \vec{k{}}_j^\tran M_j C_j=\bszero\Bigl\}\hat{f}(\bsk)^2 \var(S(\bsk))\notag\\
    &=\sum_{\bsk\in \natu_*^s}\bsone\Bigl\{\sum_{j=1}^s \vec{k{}}_j^\tran M_j C_j=\bszero\Bigl\}\hat{f}(\bsk)^2.
\end{align}
The conclusion follows by taking expectations with respect to $M$.
\end{proof}

To illustrate when the median-of-means can outperform the mean-of-means, 
suppose we can find a rare event $\ca$ such that 
$\e(\hat{\mu}_\infty\giv\ca^c)=\mu$ and 
$\var(\hat{\mu}_\infty\giv\ca^c)\ll\var(\hat{\mu}_\infty)$. 
Then if we generate independent copies of $\hat{\mu}_\infty$ and look 
at the histogram, we should see a cluster with bandwidth 
comparable to $\sqrt{\var(\hat{\mu}_\infty\giv \ca^c)}$ 
around $\mu$. Those 
$\hat{\mu}_\infty$ for which $\ca$ happens could
well be far into the tails away from $\mu$. 
In such a setting, the sample median is robust with
respect to the event $\ca$ 
and has mean square error close to $\var(\hat{\mu}_\infty\giv \ca^c)$. 
We make this intuition precise with the following lemma.

\begin{lemma} 
Let $\ca$ be an event with
$\Pr(\ca)\leq \delta$ and 
$\e(\hat{\mu}_\infty\giv \ca^c)=\mu$. Then the sample median 
$\hat{\mu}^{(r)}_{\infty}$ of  $2r-1$ independently generated values of 
$\hat{\mu}_\infty$ using a digital $(t,m,s)$-net satisfies
\begin{align}\label{eq:msebound}
    &\e((\hat{\mu}^{(r)}_{\infty}-\mu)^2)\leq {\Pr(\ca^c)}
\var(\hat{\mu}_\infty\giv \ca^c){\delta^{-1}}+(8\delta)^r \Delta^2_n
\end{align}
where
\begin{align*}\Delta_n=\min\biggl(
\omega_f(\sqrt{s}),
\frac{V_{\hk}(f)}{2^{m-t}}\sum_{i=0}^{s-1}{m-t \choose i}\biggr),
\end{align*}
$\omega_f$ gives the modulus of continuity for $f$,
and $V_{\hk}(f)$ is the total variation of $f$ in the sense of Hardy and Krause.
\end{lemma}
\begin{proof}
Conditionally on $\ca^c$, we can apply Markov's inequality to get
$$\Pr\Bigl(|\hat{\mu}_\infty-\mu|^2\geq\frac{\Pr(\ca^c)}{\delta}\var(\hat{\mu}_\infty\giv \ca^c)\!\bigm|\! \ca^c\Bigr)\leq \frac{\delta}{\Pr(\ca^c)}.$$
Hence
$$\Pr\Bigl(|\hat{\mu}_\infty-\mu|^2\geq\frac{\Pr(\ca^c)}{\delta}\var(\hat{\mu}_\infty\giv \ca^c)\Bigr)\leq \frac{\delta}{\Pr(\ca^c)} \Pr(\ca^c)+\Pr(\ca)\leq 2\delta.$$
The rest of the proof is similar to that of Corollary~\ref{cor:medianMSE}. In particular, 
$$\Pr\Bigl(|\hat{\mu}^{(r)}_{\infty}-\mu|^2\geq \frac{\Pr(\ca^c)}{\delta}\var(\hat{\mu}_\infty\giv \ca^c)\Bigr)\leq {2r-1\choose r}(2\delta)^r\leq (8\delta)^r. $$
When the `bad event' $\ca$ happens, we can bound the error in two ways: 
first it is clear that both $\hat{\mu}^{(r)}_{\infty}$ and $\mu$ are between 
$\inf_{\bsx\in [0,1]^s}f(\bsx)$ and $\sup_{\bsx\in [0,1]^s}f(\bsx)$, so their
difference is no larger than 
$\omega_f(\sqrt{s})=\sup_{\bsx\in [0,1]^s}f(\bsx)-\inf_{\bsx\in [0,1]^s}f(\bsx)$. Second, if $f$ has 
finite Hardy–Krause variation, we can apply the Koksma–Hlawka inequality \cite{hick:2014} to 
conclude that $|\hat{\mu}_\infty-\mu|\leq V_{\hk}(f)D^*_n(\bsx_0,\dots,\bsx_{n-1})$ where 
$D^*_n(\cdot)$ denotes the star discrepancy.
Because this is true for all $\hat{\mu}_\infty$, it is also true for 
$\hat{\mu}^{(r)}_{\infty}$. Because $\bsx_0,\dots,\bsx_{n-1}$ is a $(t,m,s)$-net 
regardless of the scrambling, we can apply the bound 
$$D^*_n(\bsx_0,\dots,\bsx_{n-1})\leq \frac{1}{2^{m-t}}\sum_{i=0}^{s-1}{m-t \choose i}$$
from Corollary 5.3 of \cite{dick:pill:2010}. By combining the two bounds, 
we derive $|\hat{\mu}^{(r)}_{\infty}-\mu|\leq \Delta_n$ and hence the bound 
on $\e(\hat{\mu}^{(r)}_{\infty}-\mu)^2$.
\end{proof}

The $(8\delta)^r \Delta^2_n$ term in the bound~\eqref{eq:msebound} is exponentially small in $r$ if $\delta<1/8$.
As shown in Section 3 of \cite{superpolyone}, $\var(\hat{\mu}_\infty)$ is in general $\Omega(n^{-3})$ for smooth functions $f$. Hence we only need $r\geq C_*m$ for some $C_*>0$ to make $(8\delta)^r D^2_n\ll \var(\hat{\mu}_\infty)$. With the same computational effort, the mean-of-means has variance equal to $\var(\hat{\mu}_\infty)/(2r-1)$. So heuristically, median-of-means can significantly outperform mean-of-means in terms of MSE if 
there exists an event $\ca$ such that $\Pr(\ca)\leq \delta<1/8$ and $\var(\hat{\mu}_\infty\giv \ca^c)\ll \var(\hat{\mu}_\infty)/m$. 


Motivated by Corollary~\ref{cor:concentration}, one way to choose $\ca$ is to specify a set of frequencies $K\subseteq \natu^s_*$ and let $\ca=\{\sum_{j=1}^s \vec{k{}}_j^\tran M_j C_j=\bszero \text{ for some } \bsk\in K\}$. We know that $\e(\hat{\mu}_\infty\giv \ca^c)=\mu$ because equation~\eqref{eqn:conditionallyunbiased} shows that $\hat{\mu}_\infty$ is unbiased conditionally on $\bsM=(M_1,\dots,M_s)$ and $\ca$ belongs to the $\sigma$-algebra generated by $\bsM$. Moreover, as long as $\sum_{\bsk\in K}\Pr\bigl(\sum_{j=1}^s \vec{k{}}_j^\tran M_j C_j=\bszero\bigl)\leq \delta$, we know by the union bound that $\Pr(\ca)\leq \delta$ as well. According to equation~\eqref{eqn:conditionalvariance}, 
\begin{align*}
    \Pr(\ca^c)\var(\hat{\mu}_\infty\giv \ca^c)&=\Pr(\ca^c)\e(\var(\hat{\mu}_\infty-\mu\giv \bsM)\giv \ca^c)\\
    &=\sum_{\bsk\in \natu_*^s}\Pr(\ca^c)\Pr\biggl(\,\sum_{j=1}^s \vec{k{}}_j^\tran M_j C_j=\bszero\giv \ca^c\biggr)\hat{f}(\bsk)^2\\
    &\leq \sum_{\bsk\in \natu_*^s\setminus K}\Pr\biggl(\,\sum_{j=1}^s \vec{k{}}_j^\tran M_j C_j=\bszero\biggr)\hat{f}(\bsk)^2.
\end{align*}
So in principle, if one knows all $\hat{f}(\bsk)^2$ and $\Pr\bigl(\sum_{j=1}^s \vec{k{}}_j^\tran M_j C_j=\bszero\bigl)$, then one can find a good candidate $\ca$ by solving the following combinatorial optimization problem:
\begin{equation*}
\begin{aligned}
\max_{K\subseteq \natu^s_*} \quad & \sum_{\bsk\in K}\Pr\biggl(\,\sum_{j=1}^s \vec{k{}}_j^\tran M_j C_j=\bszero\biggr)\hat{f}(\bsk)^2\\
\textrm{s.t.} \quad & \sum_{\bsk\in K}\Pr\biggl(\,\sum_{j=1}^s \vec{k{}}_j^\tran M_j C_j=\bszero\biggr)\leq \delta.\\
\end{aligned}
\end{equation*}
In particular, if $\var(\hat{\mu}_\infty)$ is dominated by a few $\bsk$ with large $\hat{f}(\bsk)^2$, then we should see a significant variance reduction after we condition on the $\ca^c$ specified by the above optimization problem.

To make the problem more tractable, we examine one case where our function $f$ is effectively low-dimensional. Suppose there are a few components $x_j$ that contribute most of the variability to $f$. More precisely, let
$$f_1(\bsx) = \e( f(\bsx)\giv x_j, j\in u)-\mu \quad\text{and}\quad f_2(\bsx) = f(\bsx)-f_1(\bsx)-\mu.$$
Then $\mu$, $f_{1}$, and $f_{2}$ are orthogonal in the 
$L^2[0,1)^s$ inner product, so that $\sigma^2(f)=\sigma^2(f_1)+\sigma^2(f_{2})$. We assume that $\sigma^2(f_1)\gg \sigma^2(f_{2})$, and then $f_1$ captures most of the variance of $f$. 

Given such a function, it is natural to choose $K=\{\bsk\in\natu_*^s\mid\supp(\bsk)\subseteq u\}$ because their associated $f(\bsk)^2$ are relatively large. Then,
Corollary~\ref{cor:tmsbound} can be strengthened in the following way:

\begin{lemma}\label{lem:finitesampletmsbound}
For non-empty $u\subseteq 1{:}s$, define $\natu^{u}\subset\natu^s$ to be the set of
$\bsk\in\natu_*^s$ with $\supp(\bsk)=u$.
Further define
\begin{align*}
t^*_u &=
m+1-\min_{\bsk\in\natu^{u}}\bigl\{ \Vert\lceil\bskappa\rceil\Vert_1
\mid
C^{\lceil\bskappa\rceil}\ \mathrm{not\ of\ full\ rank\ }
\bigr\}.
\end{align*}
If $u=\emptyset$, we conventionally define $t^*_u=0$. Then
$$\Pr\bigg(\sum_{j=1}^s \vec{k{}}_j^\tran M_j C_j=\bszero\bigg)\leq 2^{-m+t^*_{\supp(\bsk)}+|\supp(\bsk)|}.$$
\end{lemma}
\begin{proof}
The proof is basically the same as Corollary~\ref{cor:tmsbound}. By the definition of $t^*_u$, a rank-deficient $C^{\lceil\bskappa\rceil}$ must contains $m-t^*_{\supp(\bsk)}$ linearly independent rows, so
    $$\rank(C^{\lceil\bskappa\rceil-\bsone\{\supp(\bsk)\}})\geq \rank(C^{\lceil\bskappa\rceil})-|\supp(\bsk)|\geq m-t^*_{\supp(\bsk)}-|\supp(\bsk)|$$
    which proves the conclusion.
\end{proof}
\begin{remark} To compare $t^*_u$ with $t$, consider for instance a Sobol' sequence constructed by the $s$ lowest order irreducible polynomials. As shown in Section 4.5 of \cite{sobol67} 
the order of irreducible polynomials grows roughly like $\log(s)$ and $t$ is consequently $O(s\log(s))$. The supremum of $t^*_u$ on the other hand, grows no faster than $|u|\log(s)$, which is potentially much smaller than $t$. 
\end{remark}

Now we can prove the finite sample version of Corollary~\ref{cor:concentration}.
\begin{theorem}\label{thm:finitesampleconcentration}
For non-empty $u\subseteq1{:}d$, let $t_u=\max_{v\subseteq u}t^*_v$. For $\delta>0$, let $N^*_m$ be the largest integer $N$ satisfying
\begin{align*}
&2^{t^*_u+|u|}{m + |u| \choose |u|-1}N\exp\Bigl(\pi\sqrt{\frac{|u|N}{3}}\Bigr)\leq \delta 2^m,
\intertext{and}
&\sqrt{3(|u|-1)N}\leq \frac{\pi}{2}(m-t_u).
\end{align*}
Then
$$\Pr\bigg(\sum_{j=1}^s \vec{k{}}_j^\tran M_j C_j=\bszero \text{ for some } \bsk \ne\bszero \text{ with } \supp(\bsk)\subseteq u, \ \Vert\bskappa\Vert_1\leq N^*_m+m-t_u\bigg)\leq \delta.$$
\end{theorem}
\begin{proof} 
It is shown in Section 5 of \cite{pan:owen:2021:tr} that $t_v^*+|v|\leq t_u^*+|u|$ if $v\subseteq u$. So for $\bsk\ne\bszero$ with $\supp(\bsk)\subseteq u$, 
\begin{equation}\label{eqn:usubsetbound}
    \Pr\biggl(\,\sum_{j=1}^s \vec{k{}}_j^\tran M_j C_j=\bszero\biggr)\leq 2^{-m+t^*_{\supp(\bsk)}+|\supp(\bsk)|}\leq 2^{-m+t^*_u+|u|}.
\end{equation}
Now by the definition of $t^*_u$, $C^{\lceil\bskappa\rceil}$ has full rank if $\Vert\lceil\bskappa\rceil\Vert_1\leq m-t^*_{\supp(\bsk)}$, so $\Pr\big(\sum_{j=1}^s \vec{k{}}_j^\tran M_j C_j=\bszero\big)=0$. 
Let
$$
\natu_0^u = \{ \bsk\in\natu_*^s \mid  \supp(\bsk)\subseteq u\}
$$
By choosing $s=|u|$ and $R=m-t_u$ in Corollary~\ref{cor:finiteNbound2} from the appendix, we further get 
\begin{multline*}
    |\{\bsk\in\natu_0^u\mid \Vert\bskappa\Vert_1\leq N^*_m+m-t_u, \Vert\lceil\bskappa\rceil\Vert_1> m-t_u\}|\\
    <{m+|u|\choose |u|-1}N^*_m\exp\Bigl(\pi\sqrt{\frac{|u|N^*_m}{3}}\Bigr).
\end{multline*}
After taking a union bound over
all $\bsk$ in the above set, we finally get
\begin{align*}
    &\Pr\Big(\sum_{j=1}^s \vec{k{}}_j^\tran M_j C_j=\bszero \text{ for some } \bsk\ne\bszero \text{ with } \supp(\bsk)\subseteq u, \ \Vert\bskappa\Vert_1\leq N^*_m+m-t_u\Big)\\
    &\leq {m+|u|\choose |u|-1}N^*_m\exp\Bigl(\pi\sqrt{\frac{|u|N^*_m}{3}}\Bigr)2^{-m+t^*_u+|u|}\leq \delta.\qedhere
\end{align*}
\end{proof}

To interpret this result, let us consider the setting of Theorem~\ref{thm:superconvergence}. For simplicity, we will replace $f$ by $f_1(\bsx)=\e(f(\bsx)\giv x_j,j\in u)-\mu$ and pretend that the problem is $|u|$-dimensional, which is a useful approximation under our assumption on $f$. In view of Lemma~\ref{lem:Walshcoefficientbound} and equation~\eqref{eqn:conditionalvariance}, one can argue that $\var(\hat{\mu}_\infty\giv \bsM)$ is proportional to $4^{-\Vert\bskappa\Vert_1}$ for the $\bsk$ with the smallest $\Vert\bskappa\Vert_1$ among those satisfying $\sum_{j\in u} \vec{k{}}_j^\tran M_j C_j=\bszero$. This is certainly true in the asymptotic sense, as we have shown in the proof of Theorem~\ref{thm:superconvergence} that $\Vert\bskappa\Vert_0\leq \sqrt{2|u|\Vert\bskappa\Vert_1}$ and the supremum norm of partial derivatives grows no faster than $\Vert\bskappa\Vert_0!$. (More precisely, Lemma~\ref{lem:Walshcoefficientbound} only provides an upper bound on $\hat{f}(\bsk)^2$, but section 3 of \cite{superpolyone}
shows the factor $4^{-\Vert\bskappa\Vert_1}$ is in general necessary.)

By the definition of $t_u$, there exists a set of $\bsk\in\natu_*^s$ 
such that $C^{\lceil\bskappa\rceil}$ is rank-deficient and $\Vert\lceil\bskappa\rceil\Vert_1=m-t_u+1$. It is also true that $C^{\lceil\bskappa\rceil-\bsone\{\supp(\bsk)\}}$ has full rank, because otherwise $t_u$ would be even larger. Hence $\rank (C^{\lceil\bskappa\rceil-\bsone\{\supp(\bsk)\}})=\Vert\lceil\bskappa\rceil\Vert_1-|\supp(\bsk)|\leq m-t_u $ and $\Pr\bigl(\sum_{j\in u} \vec{k{}}_j^\tran M_j C_j=\bszero\bigr)\geq 2^{-m+t_u}$ 
from the second case in Lemma~\ref{lem:gaincoef}. 
On the other hand, if we condition on the 
event $\ca$ specified by Theorem~\ref{thm:superconvergence}, the 
smallest $\Vert\bskappa\Vert_1$ for which $\sum_{j\in u} \vec{k{}}_j^\tran M_j C_j=\bszero$ is 
possible is $N^*_m+m-t_u+1$ and the corresponding probability is no more 
than $2^{-m+t^*_u+|u|}$ according to equation~\eqref{eqn:usubsetbound}. So 
roughly speaking, $\var(\hat{\mu}_\infty\giv \ca^c)$ is a factor of $4^{-N^*_m}$ 
smaller than $\var(\hat{\mu}_\infty)$. In view of our previous criterion 
$\var(\hat{\mu}_\infty\giv \ca^c)$ needs to be much smaller than 
$\var(\hat{\mu}_\infty)/m$, we see that with a proper choice on the number 
of replicates, median-of-means can significantly outperform mean-of-means 
when $N^*_m\gg \log m$.

\begin{remark}
One can easily generalize the above discussion to cases where $f$ can be approximated by multiple low-dimensional functions. For instance, suppose $f$ has effective dimension $d$ in the superposition sense \cite{cafmowen}, namely $f\approx \sum_{u\subseteq 1{:}s, |u|\leq d}f_u$ where $f_u$ is the ANOVA term corresponding to subset $u$. We can define $t_d=\max_{u:|u|= d}t_u$ and $T_d=\max_{u:|u|= d}t^*_u+|u|$. By applying the above theorem to each of the ${s\choose d}$ size-$d$ subsets of $1{:}s$, we get
$$\Pr\biggl(\,\sum_{j=1}^s \vec{k{}}_j^\tran M_j C_j=\bszero \text{ for some } \bsk\ne\bszero \text{ with } |\supp(\bsk)|\leq d, \ \Vert\bskappa\Vert_1\leq N^*_m+m-t_d\biggr)\leq \delta$$
where $N^*_m$ is the largest integer $N$ satisfying
\begin{align*}
&2^{T_d}{s\choose d}{m +d \choose d-1}N\exp\Bigl(\pi\sqrt{\frac{dN}{3}}\Bigr)\leq \delta 2^m
\quad\text{and}\quad
\sqrt{3(d-1)N}\leq \frac{\pi}{2}(m-t_d).
\end{align*}
Again when $N^*_m\gg \log m$, we expect to see median-of-means outperform mean-of-means.
\end{remark}

\section{Discussion}\label{sec:discussion}

We have shown that a median-of-means strategy
based on scrambled $(t,m,s)$-nets in base $2$
can attain superpolynomial accuracy for integration
of analytic functions on $[0,1]^s$.  The main
nets we have in mind are those that arise as
the first $2^m$ points of a Sobol' sequence.
The superpolynomial rate comes with a dimension
effect that has lesser impact when the integrand
is dominated by low dimensional ANOVA components.

We have not shown that the method adapts to
lesser levels of smoothness of the integrand.
That is known to hold for $s=1$ from \cite{superpolyone}. It therefore also holds
for additive functions on $[0,1]^s$ with a rate
given by the worst smoothness of any of the 
summands.  We do not know the extent of adaptation
for more general functions.

It remains to quantify the uncertainty in
the median-of-means estimate using
the sample data.  For the mean-of-means we can
get an unbiased estimate of the variance
of the combined estimate.  
There is a central limit theorem (CLT) by Loh \cite{loh:2003} for scrambled nets as $n\to\infty$
but it only applies to nested uniform
scrambling from \cite{rtms} and is only proved
for $t=0$.
There is recent work by Nakayama and Tuffin \cite{naka:tuff:2021}
that describes CLTs for the mean-of-means
over scrambled nets as the number of replicates increases.  

For the median-of-means, things are more
complicated.  We can use nonparametric
statistical methods to get a confidence
interval for the median of $\hat\mu_{\infty,r}$
over all scrambles, but that is not the
same quantity as $\mu=\e(\hat\mu_\infty)$
and it generally depends on $m$.
There are confidence intervals for the
median-of-means (see e.g., \cite{devr:etal:2016}) but in our setting those would have width proportional to
${\var(\hat\mu_{\infty,r})^{1/2}}$.
That standard deviation does not decrease at a super-polynomial rate and so the confidence intervals would not
reflect the increased precision that comes from
using the median-of-means.
The median-of-means works so well for random
linear scrambling because that estimate is usually
very accurate apart from outliers that raise its
variance.  The presence of outliers implies that
the convergence to the Gaussian distribution
will be slow for the mean-of-means with the
random linear scrambles we study here.

This upper bound on the error has the same rate that we would
get in applying a one dimensional rule with error
$O(n^{-c\log_2(n)})$ in an $s$-fold product.
However, an $s$-fold product rule allows no nontrivial
sample sizes below $2^s$ which may be far too large to use
and still ineffective.
It is also not clear whether there would need to be $(2r-1)^s$-fold
computation in a product rule whose factors involve
medians of means.
Digital nets exist for sample sizes $2^m$ for $m\ge0$ so
we can get this rate along a practically usable sequence
of sample sizes and benefit from a good convergence
rate on the low dimensional ANOVA or other components.
The situation is similar to that in \cite{he2016extensible}
where the optimal rate under Lipschitz continuity is attained
by a grid but also by sampling along a Hilbert space-filling
curve.

\bibliographystyle{plain}
\bibliography{qmc}

\appendix
\section*{Appendix}

Here we prove some combinatorial results 
that our main theorem depends on.
We use some results from analytic
combinatorics.  Some of the standard notation
used there conflicts with that in quasi-Monte
Carlo.  For instance, both literatures study
a function denoted by $f$. Rather than change their notation to
avoid duplications, we proceed with the understanding
that some symbols have a different meaning in
this appendix than they have
in the main body of the paper. The uses
in the two settings are distinct.

We use $Q[N]$ to denote the coefficient of $x^N$ in the generating function $Q(x)$. 
We refer the reader to \cite{flaj:sedg:2009} for 
background on generating functions.

We use the bijection from the main
body of the paper between $\natu_0$ and the
set of finite cardinality subsets of 
$\natu$, denoted by $\caln$. 
Recall that for $k\in\natu_0$, we write
$k=\sum_{\ell=1}^\infty a_\ell 2^{\ell-1}$
for bits $a_\ell\in\{0,1\}$, and we set
$\kappa=\{\ell\in \natu\mid a_\ell=1\}\in\caln$.
Clearly the mapping between $k$ and $\kappa$
is a bijection between $\natu_0$ and  $\caln$. 
We extend this mapping to a bijection between
$\bsk\in \natu_0^s$ and $\bskappa\in\caln^s$ componentwise.
Therefore, combinatorial problems about $\bsk\in\natu_0^s$ can be translated into equivalent problems about $\bskappa\in\caln^s$. 

We will need a theorem of Meinardus \cite{mein:1954}.
We state the version from \cite{gran:star:erli:2008},
using the Gamma function $\Gamma(\cdot)$ and Riemann's
zeta function $\zeta(\cdot)$.
For $b_n\ge0$ let  
\begin{align}
    f(z) = \sum_{n=0}^\infty c_nz^n = \prod_{n=1}^\infty(1-z^n)^{-b_n}
\end{align}
for complex $z$ with $|z|<1$.
Meinardus' theorem will give an asymptotic expression for $c_n$.
Let
$$
D(z) = \sum_{n=1}^\infty b_nn^{-z},\quad z=\sigma +it
\quad\text{and}\quad
G(z) = \sum_{n=1}^\infty b_nz^n,\quad |z|<1.
$$
be Dirichlet and power series, respectively, for $b_n$.
\begin{theorem}[Meinardus]
  Let $b_n\ge0$ for $n\ge1$ satisfy these conditions:
  \begin{compactenum}
  \item  The Dirichlet series $D(z)$ converges in the half-plane
    $\sigma>r>0$
    and there is a constant $C_0\in(0,1]$ such that $D(z)$ for $z=\sigma+it$
    has an analytic continuation to the half-plane $\ch = \{z\mid
    \sigma\ge -C_0\}$ on which it is analytic except for a simple
    pole at $z=r$ with residue $A>0$.
  \item There is a constant $C_1>0$ such that $D(\sigma+it) = O(|t|^{C_1})$
    as $t\to\infty$ uniformly in $\sigma\ge-C_0$.
  \item There are constants $C_2>0$ and $\epsilon>0$ such that
    $g(\tau)=G(\exp(-\tau))$ for $\tau = \delta + 2\pi i\alpha$ with
    $\delta>0$
    and $\alpha\in\real$ satisfies
$\re(g(\tau))-g(\delta) \le -C_2\delta^{-\epsilon}$ for 
$|\arg(\tau)|>\pi/4$,
$0 \ne |\alpha|\le 1/2$ for small enough $\delta$.
  \end{compactenum}
  Then as $n\to\infty$,
  \begin{equation}\label{eqn:cn}
        c_n \sim C^{(1)} n^{\gamma_1}\exp\Bigl(
n^{r/(r+1)}\Bigl(1+\frac1r\Bigr)\bigl( A\Gamma(r+1)\zeta(r+1)\bigr)^{1/(r+1)}
  \Bigr)
  \end{equation}
  where 
  $$
\gamma_1 = \frac{2D(0)-2-r}{2(1+r)}
  $$
  and
  $$
C^{(1)} = e^{D'(0)}\bigl(2\pi(1+r)\bigr)^{-1/2}\bigl( A\Gamma(r+1)\zeta(r+1)\bigr)^{\gamma_2}
  $$
for
  $$
\gamma_2 = \frac{1-2D(0)}{2(1+r)}.
  $$
\end{theorem}
\begin{proof}
  This is the statement from \cite{gran:star:erli:2008}
  based on the result of \cite{mein:1954}.
  \end{proof}

\begin{theorem}\label{thm:growthrate}
For dimension $s\ge1$
$$|\{\bsk\in \natu_*^s\mid \Vert\bskappa\Vert_1=N\}|\sim\frac{C}{N^{3/4}}\exp\Bigl(\pi\sqrt{\frac{sN}{3}}\Bigr)$$
as $N\to\infty$ for some constant $C$ depending on $s$.
\end{theorem}
\begin{proof}
For $N>0$, the number of solutions $\bsk$ in $\natu_*^s$ equals the
number in $\natu_0^s$ which we study next.
By using the bijection introduced above, it suffices to bound the
number of $s$-tuples $(\Vert\kappa_1\Vert_1,\dots,\Vert\kappa_s\Vert_1)$ for which $\sum_{j=1}^s \Vert
\kappa_j\Vert_1=N$. When $s=1$, this is equal to the number of ways to
partition an integer $N$ into distinct positive integers. 
From Note I.18 of \cite{flaj:sedg:2009} that
quantity has generating function 
\begin{align}\label{eq:eulers}
    Q(x)=\prod_{n=1}^\infty (1+x^{n})=\prod_{n=1}^\infty \frac{1}{1-x^{2n-1}}.
\end{align}
For general $s$, the generating function is given by $Q^s(x)$, the $s$'th power of $Q(x)$.


Let us denote the Dirichlet series of $Q(x)$ as $D^*(z)$. To prove that $Q(x)$ has coefficients $b_n$ which satisfy the conditions of Meinardus' theorem, we first note that $Q(x)$ is also the generating function for the number of ways to
partition an integer $N$ into possibly repeated odd integers.
This equivalence is a famous result of Euler. The paper by
Bidar \cite{bida:2012} opens with a short
discussion of how Euler's observation 
follows from equation~\eqref{eq:eulers}.

Theorem 6.4 of \cite{andrews1984theory} says that Meinardus' theorem applies to the number of ways to partition an integer $N$ into sums of elements of $H_{k,a}=\{ n\in\natu\mid n=a\tmod k\}$. 
Because $Q(x)$ corresponds to the case $a=1$ and $k=2$, its coefficients satisfy those conditions. Now we can apply Meinardus' theorem to $Q(x)$ and compare equation~\eqref{eqn:cn} to the actual growth rate from Note VII.24 in \cite{flaj:sedg:2009}
$$Q[N]\sim\frac{1}{4\times3^{1/4}N^{3/4}}\exp\Bigl(\pi\sqrt{\frac{N}{3}}\Bigr).$$
Comparing to equation~\eqref{eqn:cn},
we see that the exponent of $N$ within the exponential is ${r}/(1+r)=1/2$ and that $\gamma_1=({2D^*(0)-2-r})/({2(1+r)})=-{3}/{4}$.
Therefore $r=1$ and $D^*(0)=0$.



The Dirichlet series of $Q^s(x)$ has coefficients $s b_n$, so it is equal to $sD^*(z)$. It is straightforward to verify that conditions of Meinardus' theorem still hold if all coefficients are multiplied by a positive constant, so we can apply Meinardus' theorem to $Q^s(x)$ as well. Because $s D^*(z)$ has the same pole as $D^*(z)$ and its residue is $s$ times that of $D^*(z)$, $r$ is still $1$ and $A$ is changed into $sA$. Meinardus’ theorem now gives
$$|\{\bsk\in \natu_*^s\mid \Vert\bskappa\Vert_1=N\}|=Q^s[N]\sim\frac{C}{N^{3/4}}\exp\Bigl(\pi\sqrt{\frac{sN}{3}}\Bigr)$$
for some constant $C$ depending on $s$.
\end{proof}

\begin{corollary}\label{cor:growthrate}
Let $\lambda=3\log(2)^2/\pi^2$. Then
   $$ \Bigl|\Bigl\{\bsk\in \natu_*^s\bigm| \Vert\bskappa\Vert_1\leq \frac{\lambda m^2}{s}\Bigr\}\Bigr|\sim C\frac{2^m}{\sqrt{m}}$$
   as $m\to\infty$  for some constant $C$ depending on $s$.
\end{corollary}
\begin{proof}
With slight modification, Appendix B of \cite{superpolyone} shows that
$$Q[N]\sim \frac{1}{N^{3/4}}\exp\bigl(\sqrt{\beta N}\bigr)$$
for some $\beta>0$ and also that
$$\sum_{n=1}^N Q[n]\sim \frac{1}{\beta^{1/4}N^{1/4}}\exp\bigl(\sqrt{\beta N}\bigr).$$
Hence
   $$ |\{\bsk\in \natu_*^s\mid \Vert\bskappa\Vert_1\leq N\}|\sim \Bigl(\frac{3}{\pi^2 s}\Bigr)^{1/4}\frac{C}{N^{1/4}}\exp\Bigl(\pi\sqrt{\frac{sN}{3}}\Bigr)$$
where $C$ is the constant from 
Theorem~\ref{thm:growthrate} and $\beta=\pi^2s/3$.
   The conclusion follows once we put in $N=\lfloor\lambda m^2/s\rfloor$ and notice that
   $$\exp\Bigl(\pi\sqrt{\frac{sN}{3}}\Bigr)
   \sim\exp\Bigl(\pi\sqrt{\frac{\lambda m^2}3}\Bigr)
=2^m.$$ 
\end{proof}

Next we derive some finite sample bounds using techniques from \cite{bida:2012}.
Those results give bounds for finite $N$ instead
of asymptotic equivalences as $N\to\infty$.

\begin{theorem}\label{thm:finiteNbound}
For integers $N\ge1$ and $s\ge1$
\begin{align*}
    |\{\bsk\in \natu_*^s\mid \Vert\bskappa\Vert_1=N\}|
    < \frac{\pi\sqrt{s}}{2\sqrt{3N}}\exp\Bigl(\pi\sqrt{\frac{sN}{3}}\Bigr).
\end{align*}
\end{theorem}
\begin{proof} Let $Q(x)$ and $Q^s(x)$ be the same generating functions used in Theorem~\ref{thm:growthrate}. From Lemma 3 of \cite{bida:2012}, $Q[n+1]-Q[n]\geq Q[n]-Q[n-1]$ for $n>3$. Because $Q[1]=Q[2]=1$, and $Q[3]=Q[4]=2$, it follows that $Q[n]$ is nondecreasing over integers $n\ge1$.
Because $Q^s[n]$ is given by a convolution sum of coefficients of $Q(x)$, we also see that $Q^s[n]$ is nondecreasing in n. Therefore for $0\leq x<1$,
$$Q^s(x)\geq 
\sum_{n=N}^\infty Q^s[n]x^n\ge 
Q^s[N]\sum_{n=N}^{ \infty} x^n=Q^s[N]\frac{x^N}{1-x}.$$
Furthermore, from the proof of Theorem 1 in Bidar \cite{bida:2012}, for $x=e^{-u}$ and $u>0$
$$\log(Q^s(e^{-u}))=s\log(Q(e^{-u}))<\frac{\pi^2s}{12u}$$
where we have used positivity of the dilogarithm function
at positive real arguments to obtain this bound from
Bidar's expression.

After combining the above two inequalities, we get
\begin{align*}
   \log(Q^s[N])<Nu+\frac{s\pi^2}{12u}+\log(1-e^{-u}).
\end{align*}
Now we can set $u=\sqrt{(s\pi^2)/(12N)}$ and apply the inequality $1-e^{-u}<u$, after which the above equation becomes
\begin{align*}
    \log(Q^s[N])&<\pi\sqrt{\frac{sN}{3}}+\frac{1}{2}\log\Bigl(\frac{s\pi^2}{12N}\Bigr).
\end{align*}
The conclusion follows once we exponentiate both sides.
\end{proof}

\begin{corollary}\label{cor:finiteNbound}
For integers $N\ge1$ and $s\ge1$,
\begin{align*}
    |\{\bsk\in \natu_*^s\mid \Vert\bskappa\Vert_1\leq N\}|
    <\exp\Bigl(\pi\sqrt{\frac{s(N+1)}{3}}\Bigr).
\end{align*}
\end{corollary}
\begin{proof}
Because $\exp(\pi \sqrt{s x/3})/\sqrt{x}$ is an increasing function over $[1,\infty)$, 
\begin{align*}
    |\{\bsk\in \natu_*^s\mid \Vert\bskappa\Vert_1\leq N\}|&=\sum_{n=1}^N |\{\bsk\in \natu_*^s\mid \Vert\bskappa\Vert_1= n\}|\\
    &< \int_{1}^{N+1} \frac{\pi\sqrt{s}}{2\sqrt{3x}}\exp\Bigl(\pi\sqrt{\frac{sx}{3}}\Bigr)\rd x\\
    &=\exp\Bigl(\pi\sqrt{\frac{s(N+1)}{3}}\Bigr)-\exp\Bigl(\pi\sqrt{\frac{s}{3}}\Bigr)
\end{align*}
and hence the conclusion.
\end{proof}

\begin{corollary}\label{cor:finiteNbound2}
For $R,s,N\in \natu$ satisfying $R\geq  2\sqrt{3(s-1)N}/\pi$,
\begin{align*}
    |\{\bsk\in \natu_*^s\mid \Vert\bskappa\Vert_1\leq N+R, \Vert\lceil\bskappa\rceil\Vert_1> R\}|
    <{R+s\choose s-1}N\exp\Bigl(\pi\sqrt{\frac{sN}{3}}\Bigr).
\end{align*}
\end{corollary}
\begin{proof} 
Recall that we use $\Vert \bsv\Vert_1$ 
for the sum of entries in a vector. 
There are ${n+s-1\choose s-1}$ vectors $\bsv\in\natu_0^s$ with $\Vert \bsv\Vert_1=n$. Hence
\begin{align*} 
     |\{\bsk\in \natu_*^s\mid \Vert\bskappa\Vert_1\leq N, \Vert\lceil\bskappa\rceil\Vert_1> R\}|&=\sum_{\bsv\in \natu_*^s:\Vert \bsv\Vert_1> R} |\{\bsk\in \natu_*^s\mid \Vert\bskappa\Vert_1\leq N, \lceil\bskappa\rceil=\bsv\}|\\
     &\leq \sum_{n=R+1}^N {n+s-1\choose s-1}|\{\bsk\in \natu_*^s\mid \Vert\bskappa\Vert_1\leq N-n\}|\\
     &\leq \sum_{n=R+1}^N {n+s-1\choose s-1}\exp\Bigl(\pi\sqrt{\frac{s(N-n+1)}{3}}\Bigr)
\end{align*}
where the last inequality uses the bound from Corollary~\ref{cor:finiteNbound}. The ratio of the summand with index n to the summand with index $n+1$ is 
\begin{align*}
    &{n+s-1\choose s-1}\Bigm/{n+s\choose s-1}\times \exp\Bigl(\pi\sqrt{\frac{s(N-n+1)}{3}}-\pi\sqrt{\frac{s(N-n)}{3}}\Bigr)\\
    &=\frac{n+1}{n+s}\exp\Biggl(\frac{\pi s/3 }{ \sqrt{\frac{s(N-n+1)}{3}}+\sqrt{\frac{s(N-n)}{3}}
    }\Biggr)\\
    &>\exp\biggl(-\frac{s-1}{n+1}+\frac{\pi\sqrt{s}}{2\sqrt{3(N-R)}}\biggr)
\end{align*}
where the last inequality uses
$n\ge R+1$ and
$(1+x)^{-1}>\exp(-x)$ for
$1+x = (n+1)/(n+s)$, that is
$x=(s-1)/(n+1)$. Let $N'=N-R$. If $n> R\geq  2\sqrt{3(s-1)N'}/\pi$, then the above ratio is larger than $1$ and we know that
$$\sum_{n=R+1}^N {n+s-1\choose s-1}\exp\Bigl(\pi\sqrt{\frac{s(N-n+1)}{3}}\Bigr)< {R+s\choose s-1}N'\exp\Bigl(\pi\sqrt{\frac{sN'}{3}}\Bigr)$$
which implies the conclusion.
\end{proof}
\end{document}